\documentclass[12pt,a4paper]{article}
\usepackage[latin1]{inputenc}
\usepackage[english]{babel}
\usepackage{amsmath}
\usepackage{amsthm}
\usepackage{amsfonts}
\usepackage{amssymb}
\usepackage{enumitem}
\usepackage[hmargin=2cm,paper=a4paper]{geometry}
\usepackage{mathrsfs}
\usepackage{xcolor}
    \author{Jean-Dominique Deuschel\footnote{Technische Universit\"at Berlin, {\tt deuschel@math.tu-berlin.de}},\, Henri Elad Altman\footnote{Imperial College London {\tt h.elad-altman@imperial.ac.uk}}
    and Tal Orenshtein\footnote{ Technische Universit\"at Berlin and Weierstrass Institute, {\tt orenshtein@tu-berlin.de}}}
    \title{On the gradient dynamics associated with wetting models}
    \date{}

\def\jean-dominique#1{{\color{green} Jean-Dominique: #1}}

\renewcommand{\d}{\, \mathrm{d}}

\theoremstyle{plain}
\newtheorem{thm}{Theorem}[section]

\newtheorem{prop}[thm]{Proposition}

\newtheorem{cnj}[thm]{Conjecture}
\theoremstyle{definition}

\newtheorem{rk}[thm]{Remark}
\numberwithin{equation}{section}
\begin{document}
\maketitle

\begin{abstract}
We prove a tightness result for the reversible gradient dynamics associated with critical wetting models with a shrinking strip, thus answering a conjecture of \cite{deuschel2019scaling}. We also introduce a continuous critical wetting model defined by the law of a Brownian meander tilted by its local time near the origin, and prove its convergence to the law of a reflecting Brownian motion. We further provide a description of the associated gradient dynamics in terms of an SPDE with reflection and attraction,
which we conjecture to converge to a Bessel SPDE as introduced in \cite{EladAltman2019}.
\end{abstract}



\section{Introduction}

Spectacular progress has been achieved in the past few years on scaling limits of randomly evolving discrete $(1+1)$ interface models, one celebrated example being growth models and the KPZ universality class. For interfaces evolving in the presence of a boundary (or a wall), recent development has concerned the study of the wetting models as considered in \cite{MR1865758,dgz,MR2217821} and the corresponding gradient dynamics, see \cite[Chapter 15.2]{funakistflour} and \cite{fattler2016construction,grothaus18feller}, as well as \cite{deuschel2019scaling} where a wetting model with a shrinking strip was introduced. A different approach based on integration by parts formuale was adopted in \cite{EladAltman2019} and \cite{altman2019bessel}, which proposed a family of equations, called Bessel SPDEs, which 
should describe the universality classes associated with dynamical interfaces evolving over a wall. These SPDEs, which extend to parameters $\delta <3$ the family of SPDEs considered and solved in \cite{zambotti2003integration}, can be seen as the gradient dynamics associated with the laws of Bessel processes, which they leave invariant. In particular, the Bessel SPDE of parameter $\delta=1$ admits the law of a reflecting Brownian motion as invariant measure, and can be reasonably conjectured to describe the scaling limit of dynamical critical wetting models. However, for that equation, only weak existence at equilibrium could be achieved using Dirichlet form techniques. Moreover, the associated Markov semigroup is not known to be strong Feller, so the methodology proposed in \cite{zambotti2004fluctuations} to prove scaling limits using integration by parts formulae is not effective in that case.   

In this article, we pursue a route based on approximation, by studying limits of the gradient dynamics of several wetting models. We thus consider the gradient dynamics of the wetting model with a shrinking strip of \cite{deuschel2019scaling}, for which we obtain a tightness result, but we also introduce a continuous variant of this model for which we prove a static approximation result, and consider the associated gradient SPDE. In that respect, this work is meant as a further step to make the bridge between discrete and continuous interfaces constrained by a wall.

\subsection{Wetting model with $\delta$-pinning}

We recall the definition of the wetting models we shall consider. Let us emphasise that numerous versions of these models exist: while \cite{MR1865758} considered a specific, discrete-in-space wetting model, given by a random walk with discrete steps, \cite{dgz} studied a continuous variant of it, and later the article \cite{MR2217821} addressed a general setting encompassing both the discrete and the continuous cases. Here, we restrict ourselves to the continuous-in-space case studied in \cite{dgz}, which we shall refer to as the wetting model with $\delta$-pinning. Namely, for all $N \geq 1$
and $\beta \in \mathbb{R}$, we consider the measure $\mathbb{P}^{f}_{\beta, N}$ on $\mathbb{R}_+^N$ defined by
\[\mathbb{P}^{f}_{\beta, N}({\rm{d}} \phi) = \frac{1}{Z^{f}_{\beta,N}} \rho \left( \phi \right) \prod_{i =1}^N \left( \d \phi_i \mathbf{1}_{[0,\infty)} + e^{\beta} \delta_{0}({\rm{d}} \phi_i) \right), \]
where
\[\rho(\phi) = \exp \left(- \sum_{i = 0}^{N-1} V(\phi_{i+1} - \phi_i) \right).\]
with $\phi(0) := 0$, and where $V: \mathbb{R} \to \mathbb{R} \cup \{ \infty \}$ is such that $\exp(-V)$ is continuous, $V(0)< \infty$, and
\[ \int_\mathbb{R} e^{-V} < \infty. \]
In this article, for simplicity, we shall consider $V(x) = \frac{x^{2}}{2}$ for all $x \in \mathbb{R}$. Above, the superscript $f$ stands for ``free'', meaning that we do not constrain the value of $\phi_N$.

We also recall the main theorem in \cite{dgz} (re-derived later in \cite{MR2217821} with a simplified proof). Let $H= L^2(0,1)$. For all $N \geq 1$, we  let $\Phi_{N} : \mathbb{R}^{N} \to H$ denote the rescaling and interpolation map defined by
\begin{equation}
\label{scaling_and_interp_map}
\Phi_{N} (\phi) (y) = \frac{1}{\sqrt{N}} \phi_{\lfloor N y \rfloor} +
\frac{1}{\sqrt{N}} (Ny - \lfloor N y \rfloor) \left( \phi_{\lfloor N y \rfloor + 1} - \phi_{\lfloor N y \rfloor} \right), \quad y \in [0,1],
\end{equation}
where we use the convention that $\phi_0 := 0$ in the right-hand side above. We will denote by $H_{N}$ the vector space $\Phi_{N}(\mathbb{R}^{N}) \subset H$, which coincides with the space of continuous piecewise affine functions adapted to the partition $[\frac{i-1}{N}, \frac{i}{N}), 1 \leq i \leq N$. We finally denote by $\mathbf{P}^{f}_{\beta, N}$ the image of the measure $\mathbb{P}^{f}_{\beta, N}$ under $\Phi_{N}$. 

\begin{thm}[\cite{dgz}]
There exists $\beta_c \in (0,\infty)$ such that:
\begin{itemize}
\item if $\beta < \beta_c$ (subcritical case), then $\mathbf{P}^{f}_{\beta, N} \underset{N \to \infty}{\longrightarrow} m$ in law, where $m$ is the law of a Brownian meander on $[0,1]$
\item if $\beta = \beta_c$ (critical case), then $\mathbf{P}^{f}_{\beta, N} \underset{N \to \infty}{\longrightarrow} P^1_0$ in law, where $P^1_0$ is the law of a reflecting Brownian motion started from $0$ on $[0,1]$
\item if $\beta > \beta_c$ (supercritical case), then $\mathbf{P}^{f}_{\beta, N}$ converges in law, as $N \to \infty$, to the measure concentrated on the function identically equal to $0$ on $[0,1]$.
\end{itemize}
\end{thm}

One can ask whether it is possible to build a Markov process on $\mathbb{R}_+^N$ admitting the measure $\mathbb{P}^{f}_{\beta, N}$ as a reversible measure. Due to the presence of Dirac masses in the definition of $\mathbb{P}^{f}_{\beta, N}$, this problem is highly non-trivial. For instance, in the particular case $N=1$ and $\rho \equiv 1$, a natural candidate is given by a sticky reflecting Brownian motion, which is a solution to
\[\begin{split}
\left\{ \begin{array}{ll}
{\displaystyle
d X_t = \frac{1}{2} \d \ell^0_t + \mathbf{1}_{\{X_s > 0 \}} \d B_s
}
\\ \\
\mathbf{1}_{\{X_t = 0 \}} \d t = \frac{e^\beta}{2} \d \ell^0_t.
\end{array} \right.
\end{split}\]
This stochastic equation is known to possess a weak solution, but no strong solutions (see e.g.\ \cite{engelbert2014stochastic}). In \cite{fattler2016construction} and \cite{grothaus18feller} some diffusions having $\mathbb{P}^{f}_{\beta, N}$ as a reversible measure were constructed and studied using sophisticated Dirichlet form methods.

\subsection{Wetting model with a shrinking strip}

A variant of the above wetting model was introduced in \cite{deuschel2019scaling}. Namely, for all $N \geq 1$, and $a>0$, one considers the measure $\mathbb{P}^{f}_{\varphi_{a}, N}$ on $\mathbb{R}_+^N$ defined by
\[\mathbb{P}^{f}_{\varphi_{a}, N}({\rm{d}} \phi) = \frac{1}{Z^{f}_{\varphi_{a},N}} \rho(\phi) \prod_{i =1}^N e^{\varphi_{a}(\phi_i)} \d \phi_i \, , \]
where $\rho$ is as before, and where, for all $a>0$, $\varphi_a: \mathbb{R}_+ \to \mathbb{R}_+$ is a smooth function supported in $[0,a]$. Thus, the measure $\mathbb{P}^{f}_{\beta, N}$ above, which had some atoms, has been replaced by a measure which is absolutely continuous w.r.t.\ the Lebesgue measure on $\mathbb{R}_+^N$. With this new version with a ``strip'', one can however recover a scaling limit result as in the critical regime above.

Let us denote by $\mathbf{P}^{f}_{\varphi_{a}, N}$ the image of $\mathbb{P}^{f}_{\varphi_{a}, N}$ under the map $\Phi_N$. Then,  $\mathbf{P}^{f}_{\varphi_{a_N}, N}$ converges in law, as $N \to \infty$, to the law of a reflecting Brownian motion on $[0,1]$, whenever the pinning functions $(\varphi_a)_{a>0}$ satisfy Condition (A) in \cite{deuschel2019scaling} and $a_N = o(N^{-1/2})$ \cite[Theorem 1.5]{deuschel2019scaling}. Here and in the sequel, we fix such $(\varphi_{a})_{a>0}$ and a sequence $(a_N)_{N \geq 1}$. 
Our aim is to show a tightness result for the dynamics associated with $\mathbb{P}^{f}_{\varphi_{a_N}, N}$, $N \geq 1$.

The advantage of this new model with respect to the $\delta$-pinning measure is the fact that $\mathbb{P}^{f}_{\varphi_{a}, N}$ is absolutely continuous w.r.t.\ the Lebesgue measure on $\mathbb{R}_+^N$, so it is straightforward to construct an associated reversible Markov process. It suffices indeed to consider the corresponding gradient SDE
\begin{equation}
\label{sde_pinning_strip}
\begin{split}
\left\{ \begin{array}{ll}
{\displaystyle
X_i(t) = - \int_0^t \partial_i \mathcal{H}_N (X(s)) \d s + \ell^i_t + \sqrt{2} W^i_t, \quad i =1, \ldots, N}
\\ \\
X_i(t) \geq 0, \, \d \ell^i_t \geq 0, \, \int_0^\infty X_i(t) \d \ell^i_t = 0,
\end{array} \right.
\end{split}
\end{equation}
where $W^1,\ldots, W^N$ are independent Brownian motions, and with a random initial condition $X_0$ distributed as $\mathbb{P}^{f}_{\varphi_{a}, N}$. Above, we have denoted by $\mathcal{H}_N: \mathbb{R}_+^N \to \mathbb{R}$ the potential defined by:
\[ e^{- \mathcal{H}_N(\phi)} = \rho(\phi) \prod_{i=1}^{N} e^{\varphi_{a_N}(\phi_{i})}, \quad \phi \in \mathbb{R}_+^N.\]
It was conjectured in \cite{deuschel2019scaling} that properly rescaled, the sequence of processes $(X_i(t))_{1\le i\le N, t\ge 0}$ given by \eqref{sde_pinning_strip} is tight. In this article, we prove this claim using the Lyons-Zheng decomposition.  After proving tightness and hence achieving a limiting dynamics, a natural step forward might be the following. Since the discrete wetting model of [DO19] is specific while the conjectured dynamics is independent of the pinning mechanism details, one should expect a more robust construction of the static model to which more tools could apply. Having that in mind, we introduce in Chapter 4.2 a continuous version of the strip wetting model. This new model is given by a path measure which is absolutely continuous with respect to the law of a $3$-dimensional Bessel process, with an explicit Radon-Nikodym derivative involving the local times of the latter at a level $\eta>0$. More precisely, we consider the measure 
\[P^{1,\eta}_{a}({\rm{d}} X) = \frac{X_{1} \wedge \eta}{X_{1}} \frac{a}{a \wedge \eta} \exp \left( \frac{1}{2 \eta} L^{\eta}_{1} \right) \, P^{3}_{a} ({\rm{d}} X),\]
where  $P^{3}_{a}$ denotes the law of a $3$-dimensional Bessel process started from $a$ on $[0,1]$, and $L^{\eta}_{1}$ denotes the value, at time $1$, of its local time at the level $\eta$. This measure, as it turns out, corresponds to the law of the unique strong solution on $[0,1]$ to the SDE
$$X_t = a + \int_{0}^{t} \frac{\mathbf{1}_{\{X_{s} \leq \eta\}}}{X_{s}} \d s + B_{t},$$
which corresponds to the SDE of a $3$-Bessel process, but with a truncation at $\eta$ in the drift. We show convergence of this measure to the law $P^{1}_{a}$ of a $1$-dimensional Bessel process started from $a$, as $\eta \to 0$. This result is a continuum version of the static scaling limit theorem of [DGZ05], [CGZ06], and [DO19]. The advantage of considering such a model is that the associated convergence result is much easier to prove than for its discrete counterparts. Moreover, the measures $P^{1,\eta}_{a}$ provide a monotonously continuous interpolation between the probability measures $P^{3}_{a}$ (which corresponds formally to $P^{1,\infty}_{a}$) and $P^{1}_{a}$ (corresponding formally to $P^{1,0}_{a}$), a result interesting in its own respect. Finally, the gradient dynamics associated with a mollified version of $P^{1,\eta}_{a}$ posess an interesting interpretation in terms of an SPDE with reflection combined with an attractive mechanism, which we conjecture to converge to the same limit as the discrete dynamical model.

\subsection{Structure of this paper}

In Section \ref{sec_tightness} of this paper we prove tightness for the sequence of processes $(X_i(t))_{1\le i\le N, t\ge 0}$ given by \eqref{sde_pinning_strip}. In Section \ref{sec_ibpf} we write the associated integration by parts formula, which motivates a conjecture on the corresponding limit in law. In Section \ref{sec_continuum} we introduce a continuous version of the wetting model and prove that it converges towards the law of the modulus of a Brownian motion. We also write the associated gradient dynamics in terms of an SPDE with reflection combined with an attractive mechanism, and formulate a conjecture for the corresponding limit in law.

\section{A tightness result}
\label{sec_tightness}

In  Section 1.5 of \cite{deuschel2019scaling}, the first and third author considered the processes $(Y^N_t)_{t \geq 0}, N \geq 1$, 
where $Y^{N}_{t} = \Phi_{N} (X(N^2t))$, $(X(t))_{t \geq 0}$ is the reversible evolution in $\mathbb{R}_+^N$ for the pinning measure $\mathbb{P}^{f}_{\varphi_{a}, N}$ as given by \eqref{sde_pinning_strip}, and $\Phi_{N}$ as in \eqref{scaling_and_interp_map} above.
%
Here and below we shall denote the inner product on $H=L^2(0,1)$ by $\langle \cdot, \cdot \rangle$, and set
\[ K := \{h \in H, \quad h \geq 0 \quad \text{a.e.} \}. \]
For all $\gamma > 0$, we introduce the space $H^{-\gamma}(0,1)$, completion of $H=L^{2}(0,1)$ w.r.t.\ the norm $\| \cdot\|_{-\gamma}$ defined by
\[ \|f \|^{2}_{-\gamma} := \sum_{n=1}^{\infty} n^{-2\gamma} | \langle f, e_{n} \rangle |^{2}, \qquad f \in H, \]
where 
\begin{equation}
\label{def_e_n}
e_{n}(\theta) :=\sqrt{2} \sin(n \pi \theta), \qquad \theta \in [0,1]
\end{equation}

\begin{thm}
\label{thm_tightness_discrete_wetting}
For all $T>0$, the family of processes $(Y^{N}_{\cdot})_{N \geq 1}$ is tight in $C([0,T],H^{-1}(0,1))$.
\end{thm}

\begin{proof}
Recall that 
$\Phi_N: \mathbb{R}^N \to H$ is the function defined by \eqref{scaling_and_interp_map}, and $H_N = \Phi_N(H)$. Let us denote by $(x^{i}_{N})_{1\leq i \leq N}$ the image of the canonical basis of $\mathbb{R}^N$ under $\Phi_{N}$. It then follows that, for all $N \geq 1$, the process $\left( Y^{N}_{t} \right)_{t \geq 0}$ coincides in law with the reversible Markov process associated with the Dirichlet form $(\text{Dom}(\mathcal{E}^{f}_{\varphi_{a_N}, N}), \mathcal{E}^{f}_{\varphi_{a_N}, N})$ which is the closure of the form
\[ \mathcal{E}^{f}_{\varphi_{a_N}, N}(u,v) = N^{2} \int_{K_{N}} \sum_{i=1}^{N} \langle \nabla u(x), x^{i}_{N} \rangle \langle \nabla v(x), x^{i}_{N} \rangle \d \mathbf{P}^{f}_{\varphi_{a_N}, N}(x),  \quad u,v \in C^{1}_{b}(H_{N}), \]
where $K_{N} := \{ x \in H_{N}, \, x \geq 0 \}$  and $\mathbf{P}^{f}_{\varphi_{a_N}, N}$ is the image of the measure $\mathbb{P}^{f}_{\varphi_{a_N}, N}$ under $\Phi_{N}$. Then, for all $T>0$ and $h \in H$, by the Lyons-Zheng decomposition , see e.g.\ Thm 5.7.1 in \cite{fukushima2010dirichlet}, we have
\[ \langle Y^{N}_{t}, h \rangle - \langle Y^{N}_{0}, h \rangle = \frac{1}{2} M^{1}_{t} - \frac{1}{2} \left(M^{2}_{T} - M^{2}_{T-t} \right), \]
where $M^{i}$ is an $H_{N}$-valued $(\mathcal{F}^{i}_{t})_{t \geq 0}$ martingale, $\mathcal{F}^{1}_{t} = \sigma (Y^{N}_{s}, s \leq t)$ and $\mathcal{F}^{2}_{t} = \sigma (Y^{N}_{T-s}, s \leq t)$. More precisely, defining $\psi \in C^{1}_{b}(H_{N})$ by $\psi:= \langle h , \cdot \rangle$, by Theorem 5.7.1 in \cite{fukushima2010dirichlet}, we have the above decomposition with
\[M^{1}_{t} := M^{[\psi]}_{t}, \quad M^{2}_{t}(\omega) := M^{[\psi]}_{t}(r_{T} \omega), \]
where $r_{T}$ is the time-reversing operator on the canonical space $\Omega := C([0,T], K_N)$:
\[ (r_{T} \omega)_{t} = \omega_{T-t}, \qquad \omega \in \Omega, \, t \in [0,T]. \]
Moreover, $M^{[\psi]}$ denotes the martingale additive functional appearing in the Fukushima decomposition of the continuous additive functional (CAF) given by
\[A^{[\psi]}_{t} := \psi(Y^{N}_{t}) - \psi(Y^{N}_{0}) = \langle Y^{N}_{t}, h \rangle - \langle Y^{N}_{0}, h \rangle, \quad t \geq 0, \]
see Theorem 5.2.2. Hence, the quadratic variation of the martingales $M^1$ and $M^2$ is given by the sharp bracket $\langle M^{[\psi]} \rangle_{t}$ of the martingale additive functional $M^{[\psi]}$. The latter is a positive continuous additive functional with Revuz measure $\mu_{\psi}$ satisfying
\[ \int_{K_{N}} f(x)  \d \mu_{\psi}(x) = 2  \mathcal{E}^{f}_{\varphi_{a_N}, N}(\psi f, \psi) -  \mathcal{E}^{f}_{\varphi_{a_N}, N}(\psi^{2},f), \]
for all $f \in \text{Dom}(\mathcal{E}^{f}_{\varphi_{a}, N})$, see Theorem 5.2.3 in \cite{fukushima2010dirichlet}. But, for all $f \in C^{1}_{b}(H_{N})$, by the Leibniz rule, and recalling that $\nabla \psi =h$, we have
\[2  \mathcal{E}^{f}_{\varphi_{a_N}, N}(\psi f, \psi) -  \mathcal{E}^{f}_{\varphi_{a_N}, N}(\psi^{2},f) = 2 N^{2} \int_{K} f(x) \sum_{k=1}^{N} \langle h, x^{N}_{i} \rangle^{2} \, \mathbf{P}^{f}_{\varphi_{a_N}, N}({\rm{d}} x).\]
Therefore, for all $f \in C^{1}_{b}(H_{N})$, it holds
\[ \int_{K} f(x) \mu_{\psi}({\rm{d}} x) = \int_{K} f(x) \left(2 N^{2} \sum_{k=1}^{N} \langle h, x^{N}_{i} \rangle^{2} \right) \mathbf{P}^{f}_{\varphi_{a_N}, N}({\rm{d}} x). \]
Therefore, we deduce that
\[\mu_{\psi}({\rm{d}} x) =  \left(2 N^{2} \sum_{k=1}^{N} \langle h, x^{N}_{i} \rangle^{2} \right) \mathbf{P}^{f}_{\varphi_{a_N}, N}({\rm{d}} x).\]
Hence, by the Revuz correspondence, we deduce the equality
\[ \langle M^{[\psi]} \rangle_{t} = \left(2 N^{2} \sum_{k=1}^{N} \langle h, x^{N}_{i} \rangle^{2} \right) \, t, \quad t \geq 0 \]
in the sense of additive functionals, which implies that, for all $i=1,2$
\begin{equation}
\label{express_qv}
\langle M^{i} \rangle_{t} = 2 N^{2} \sum_{k=1}^{N} \langle h, x^{N}_{i} \rangle^{2} \, t, \quad t \geq 0.
\end{equation}
But recall that, for all $k=1, \ldots, N$, $x^{N}_{k}=\Phi_{N} (e_{k})$, where $(e_{1}, \ldots, e_{N})$ is the canonical basis of $\mathbb{R}^{N}$. In words, $x^{N}_{k}$ is the function in $H_{N}$ which takes the value $\frac{1}{\sqrt{N}}$ at the point $k/N$, and the value $0$ at the points $j/N$, $j \neq k$. Note in particular that
\[ 0 \leq x^{N}_{k} \leq \frac{1}{\sqrt{N}} \mathbf{1}_{[\frac{k-1}{N}, \frac{k+1}{N}]}, \]
so that we have the bound $\|x^{N}_{k}\|^{2} \leq 2 N^{-2}$. Using the Cauchy-Schwarz inequality followed by the latter bound, we deduce from \eqref{express_qv} that
\[\langle M^{i} \rangle_{t} \leq 4 ||h||^{2} \, t. \]
Hence, by the BDG inequality, for all $p \geq 1$, there exists a constant $C_{p}>0$ (depending only on $p$) such that, for all $t \geq s \geq 0$
\[ \left( \mathbb{E} \left[ \langle Y^{N}_{t} - Y^{N}_{s} , h \rangle^{p} \right] \right)^{1/p} \leq C_{p} (t-s)^{1/2} \, ||h||\]
The above being true for any $h \in H$, we deduce that, for all $p \geq 2$
\[ \begin{split}
\left( \mathbb{E} \left[\|Y^{N}_{t} - Y^{N}_{s} \|^{p}_{H^{-1}(0,1)} \right] \right)^{1/p} &= \left( \mathbb{E} \left[ \left( \sum_{k=1}^{\infty} \langle Y^{N}_{t} - Y^{N}_{s}, e_{k} \rangle^{2} k^{-2} \right)^{p/2} \right] \right)^{1/p} \\
& \leq \zeta(2)^{\frac{1}{2}-\frac{1}{p}} \left( \mathbb{E} \left[ \sum_{k=1}^{\infty} \langle Y^{N}_{t} - Y^{N}_{s}, e_{k} \rangle^{p} k^{-2} \right] \right)^{1/p} \\
&\leq \zeta(2)^{1/2} C_{p} (t-s)^{1/2} \\
& \leq C'_{p} (t-s)^{1/2},
\end{split}\]
where we applied Jensen's inequality to obtain the second line, and $C'_{p} = \zeta(2)^{1/2} C_{p}$, with $\zeta(2) = \sum_{k =1}^\infty k^{-2}$. We have thus obtained, for all $p \geq 2$, the following bound holding uniformly in $t,s \in [0,T]$
\[ \left( \mathbb{E} \left[\|Y^{N}_{t} - Y^{N}_{s} \|^{p}_{H^{-1}(0,1)} \right] \right)^{1/p} \leq C'_{p} |t-s|^{1/2}. \]
Moroever, for all $t \geq 0$, $Y^N_t \overset{(d)}{=} \mathbf{P}^{f}_{\varphi_{a_N}, N}$ and, by Theorem 1.5 in \cite{deuschel2019scaling}, $\mathbf{P}^{f}_{\varphi_{a_N}, N} \underset{N \to \infty}{\longrightarrow} P^{1}_{0}$ in law, where $P^{1}_{0}$ is the law of a reflecting Brownian bridge started from $0$ on $[0,1]$. Hence, since $H^{-1}(0,1)$ is a Polish space, invoking  \cite[Thm 7.2, Chap. 3]{MR838085}, we deduce that the sequence of processes $(Y^{N}_{t})_{t \in [0,T]}, N \geq 1$ is tight in $C([0,T], H^{-1}(0,1))$.
\end{proof}

\section{An integration by parts formula and conjecture for the scaling limit}

\label{sec_ibpf}

Let $C^{1}_{b} ( \mathbb{R}^N)$ be the set of continuously differentiable functions on $\mathbb{R}^{N}$ with bounded derivatives. For any $f \in C^{1}_{b} ( \mathbb{R}^N)$ and $h \in \mathbb{R}^{N}$,
we denote by $\partial_{h} f$ the derivative of $f$ in the direction $h$
\[ \partial_{h} f(\phi) := \underset{\epsilon \to 0}{\lim} \, \frac{f(\phi+\epsilon h) - f(\phi)}{\epsilon}, \quad \phi \in \mathbb{R}^{N}. \]
We then have the following integration by parts formula (IbPF) for the measure $\mathbb{P}^{f}_{\varphi_{a}, N}$ on $\mathbb{R}^{N}$:

\begin{prop}
\label{prop_ibpf_pinning}
For all $f \in C^{1}_{b} ( \mathbb{R}^{N})$ and $h \in \mathbb{R}^{N}$, we have
\begin{equation}
\label{ibpf_setting_strip} 
\begin{split}
 \int_{\mathbb{R}_{+}^{N}} \partial_{h} f(\phi) \, \mathbb{P}^{f}_{\varphi_{a}, N}({\rm{d}} \phi) &= \sum_{i=1}^{N} h_{i} \left( \int_{0}^{a} e^{\varphi_{a}(b)} \frac{\d}{\d b} \sigma^{N}_{i} (f|b) \d b   -  \sigma^{N}_{i} (f|a) \right) \\
& -  \int_{\mathbb{R}_{+}^{N}} f(\phi)  \sum_{i=1}^{N} \phi_{i} \, \left( h_{i+1} + h_{i-1} - 2 h_{i} \right) \, \mathbb{P}^{f}_{\varphi_{a}, N}({\rm{d}} \phi).
\end{split},  
\end{equation}
where $h_0=h_{N+1}=0$ and for all $b \geq 0$ and $i=1,\ldots,N$, $\sigma^{N}_{i} (f|b) = \int_{\mathbb{R}_{+}^{N}} f(\phi) \sigma^{N}_{i} ({\rm{d}} \phi|b)$, where,
\[ \begin{split}
\sigma^{N}_{i} ({\rm{d}} \phi |b)
&:= e^{-\varphi_{a}(b)} \,  \frac{\mathbb{P}^{f}_{\varphi_{a}, N}(\phi_{i} \in \d b)}{\d b} \, \mathbb{P}^{f}_{\varphi_{a}, N} \left( {\rm{d}} \phi | \phi_{i} = b \right) \\
&= \frac{1}{Z^{f}_{\varphi_{a},N}} \rho(\phi) \delta_{b}({\rm{d}} \phi_{i}) \prod_{n \neq i} e^{\varphi_{a}(\phi_{n})} \d \phi_{n},
\end{split}\]   
or, formally written 
\[ \sigma^{N}_{i} ({\rm{d}} \phi |b) = e^{-\varphi_{a}(b)} \mathbf{1}_{\{b\}}(\phi_i) \, \mathbb{P}^{f}_{\varphi_{a}, N}({\rm{d}} \phi). \]
\end{prop}
Above $\frac{\mathbb{P}^{f}_{\varphi_{a}, N}({\phi_{i}} \in \d b)}{\d b} $ denotes the density at $b$ of the law of $\phi_i$ (under the measure $\mathbb{P}^{f}_{\varphi_{a}, N}$) with respect to the Lebesgue measure on $\mathbb{R}_+$.  Heuristically, the measures $\sigma^{N}_{i}({\rm{d}} \phi | b)$ are meant to be a discrete analog of the measures $\Sigma^{1}_{r} ({\rm{d}} X | b)$, $r \in (0,1), b \geq 0,$ introduced in \cite{EladAltman2019}, see Def.\ 3.4 therein. In particular we remark some resemblance of the above IbPF with formula (4.4) of \cite{EladAltman2019}. Indeed, the first term in the right-hand side of \eqref{ibpf_setting_strip} only involves values of $\sigma^{N}_{i}( f | b)$ for $b$ close to $0$ (namely $b \leq a$), and is somewhat reminiscent of the term
\[ \int_0^1 \d r h_r \frac{\d^2}{\d a^2} \Sigma^{1}_{r} (\Phi(X) | a) \biggr \rvert_{a=0} \]
in the right-hand side of (4.4) in \cite{EladAltman2019}.

\begin{proof}
Recalling the definition of $\mathbb{P}^{f}_{\varphi_{a}, N}$, and integrating by parts with respect to the Lebesgue measure on $\mathbb{R}_+^N$, we have
\[\begin{split}
\int_{\mathbb{R}_{+}^{N}} \partial_{h} f(\phi) \, \mathbb{P}^{f}_{\varphi_{a}, N}({\rm{d}} \phi) = &- \frac{1}{Z^{f}_{\varphi_{a},N}} \int_{\mathbb{R}_{+}^{N}} f(\phi) \, \partial_{h} \rho(\phi) \prod_{n=1}^{N} e^{\varphi_{a}(\phi_{n})} \d \phi_{n} \\
&- \sum_{i=1}^{N} h_i \frac{1}{Z^{f}_{\varphi_{a},N}} \int_{\mathbb{R}_{+}^{N}} f(\phi) \rho(\phi) e^{\varphi_{a}(0)} \delta_{0}({\rm{d}} \phi_i) \prod_{\substack{n=1 \\ n \neq i}}^{N} e^{\varphi_{a}(\phi_{n})} \d \phi_{n} \\
&- \sum_{i=1}^{N} h_i \frac{1}{Z^{f}_{\varphi_{a},N}} \int_{\mathbb{R}_{+}^{N}} f(\phi) \rho(\phi) \varphi_{a}'(\phi_i) \prod_{n=1}^{N} e^{\varphi_{a}(\phi_{n})} \d x_{n}.
\end{split}\]
We recognize in the first term of the right-hand side above the quantity
\[-  \int_{\mathbb{R}_{+}^{N}} f(\phi)  \sum_{i=1}^{N} \left( h_{i+1} + h_{i-1} - 2 h_{i} \right) \phi_{i} \, \mathbb{P}^{f}_{\varphi_{a}, N}({\rm{d}} \phi)\]
On the other hand, the second term can be rewritten
\[ - \sum_{i=1}^{N} h_i \, e^{\varphi_{a}(0)} \, \sigma^{N}_{i}(f|0). \]
Finally, the third term can be rewritten
\[ - \sum_{i=1}^{N} h_i \int_0^a \frac{d}{db} \left( e^{\varphi_{a}(b)} \right) \sigma^{N}_{i}(f|b) \, \d b, \]
or, after an integration by parts:
\[ - \sum_{i=1}^{N} h_i \left\{\sigma^{N}_{i}(f|a) - e^{\varphi_{a}(0)} \, \sigma^{N}_{i}(f|0) -  \int_0^a  e^{\varphi_{a}(b)} \frac{d}{db} \left(\sigma^{N}_{i}(f|b) \right)  \, \d b \right\}. \]
Adding up all three quantities, and noting the cancellation of $\sum_{i=1}^{N} h_i \, e^{\varphi_{a}(0)} \, \sigma^{N}_{i}(f|0)$,
we obtain the claim.
\end{proof}

Thanks to the above discrete IbPF, we obtain an IbPF for the rescaled wetting dynamics on the space of paths $H$. 
For convenience of the presentation, and with abuse of notation, we assume here that $\mathbf{P}^{f}_{\varphi_{a}, N}$ is the image of the measure 
$\mathbb{P}^{f}_{\varphi_{a}, N}$ under $\tilde\Phi_N$, the left-continuous with right-hand limits (or c\`agl\`ad) piecewise constant interpolation modification of $\Phi_N$: 
$\tilde\Phi_N(\phi)(y):=\frac{1}{\sqrt{N}}\phi_{\lceil N y \rceil}$.
We denote by $\Pi_N : H \to \tilde H_N$ the orthogonal projection of $H$ onto the space $\tilde H_N = \tilde\Phi_N(\mathbb{R}^N)$, and by $C^{1}_{b} (H)$ be the set of continuously Fr\'{e}chet differentiable functions on $H$ with bounded Fr\'{e}chet differential.

\begin{prop}
\label{prop_ibpf_pinning_path}
For all $\Psi \in C^{1}_{b} (H)$ and $h \in C^2_c(0,1)$, we have
\begin{equation}
\label{ibpf_pinning_path}
\begin{split}
 \int_{H} \partial_{\Pi_N h} \Psi(\zeta) \, \mathbf{P}^{f}_{\varphi_{a}, N}({\rm{d}} \zeta) &= 
  \int_{0}^1 (\Pi_N h)_{y} D^N_{\Psi}(y)\d y
 -  \int_{H} \Theta^N_h(\zeta) \, \Psi(\zeta) \, \mathbf{P}^{f}_{\varphi_{a}, N}({\rm{d}} \zeta),
\end{split}
\end{equation}
where 
\[
D^N_{\Psi}(y)= N^{3/2}\left( \int_{0}^{a} e^{\varphi_{a}(b)} \frac{\d}{\d b} \sigma^{N}_{\lceil yN \rceil} (\Psi \circ {\tilde{\Phi}}_N |b) \d b   -  \sigma^{N}_{\lceil yN \rceil} (\Psi \circ {\tilde{\Phi}}_N |a) \right)
\]
and for all $\zeta \in H$
\[ \Theta^N_h(\zeta) = N^2 \int_0^1 \zeta_{y} \, \Delta^N_h(y) \d y\]
for 
\[
\Delta^N_h(y) =  h_{y+\frac{1}{N}} \mathbf{1}_{\{y\le 1-\frac{1}{N}\}} + h_{y-\frac{1}{N}}\mathbf{1}_{\{y\ge \frac{1}{N}\}} - 2 h_y. 
\]
\end{prop}

\begin{proof}
First note that 
\[(\Pi_N h)_y = N \int_{\frac{\lceil yN \rceil - 1}{N}}^{\frac{\lceil yN \rceil}{N}}h(u) \, \d u, \qquad y \in [0,1]. \]
In particular $\Pi_N h = \tilde{\Phi}_N(\bar{h})$, where $\bar{h} \in \mathbb{R}^N$ is given by 
\[\bar{h}_i = N^{3/2} \int_{\frac{i-1}{N}}^{\frac{i}{N}} h(u) \, \d u, \qquad i=1, \ldots, N.\]  
Now, by the IbPF \eqref{ibpf_setting_strip} applied to the functional $\Psi \circ \tilde{\Phi}_N$ and the vector $\bar{h}$, we have
\[\begin{split}
 \int_{\mathbb{R}_{+}^{N}} \partial_{\bar{h}} (\Psi \circ \tilde{\Phi}_N)(\phi) \, \mathbb{P}^{f}_{\varphi_{a}, N}({\rm{d}} \phi) &= \sum_{i=1}^{N} \bar{h}_{i} \left( \int_{0}^{a} e^{\varphi_{a}(b)} \frac{\d}{\d b} \sigma^{N}_{i} (\Psi \circ \tilde{\Phi}_N|b) \d b   -  \sigma^{N}_{i} (\Psi \circ \tilde{\Phi}_N|a) \right) \\
& -  \int_{\mathbb{R}_{+}^{N}} \Psi (\tilde{\Phi}_N(\phi))  \sum_{i=1}^{N} \phi_{i} \, \left( \bar{h}_{i+1} + \bar{h}_{i-1} - 2 \bar{h}_{i} \right) \, \mathbb{P}^{f}_{\varphi_{a}, N}({\rm{d}} \phi).
\end{split},\]  
with the convention $\bar{h}_0=\bar{h}_{N+1}=0$. Since $\tilde{\Phi}_N(\bar{h})=\Pi_N h $, $\partial_{\bar{h}} (\Psi \circ \tilde{\Phi}_N) = \partial_{\Pi_N h} \Psi$. Therefore
\[\int_{\mathbb{R}_{+}^{N}} \partial_{\bar{h}} (\Psi \circ \tilde{\Phi}_N)(\phi) \, \mathbb{P}^{f}_{\varphi_{a}, N}({\rm{d}} \phi)  = \int_{H} \partial_{\Pi_N h} \Psi(\zeta) \, \mathbf{P}^{f}_{\varphi_{a}, N}({\rm{d}} \zeta). \]
On the other hand, by the definition of $\bar{h}$,
\[\sum_{i=1}^{N} \bar{h}_{i} \left( \int_{0}^{a} e^{\varphi_{a}(b)} \frac{\d}{\d b} \sigma^{N}_{i} (\Psi \circ \Phi_N|b) \d b   -  \sigma^{N}_{i} (\Psi \circ \Phi_N|a) \right) = \int_{0}^1 (\Pi_N h)_{y} D^N_{\Psi}(y)\d y. \]
Finally, for the last term, noting that
\[ \sum_{i=1}^{N} \phi_{i} \, \left( \bar{h}_{i+1} + \bar{h}_{i-1} - 2 \bar{h}_{i} \right) =  N^2 \int_0^1 \tilde{\Phi}_N(\phi)_y \, \Delta^N_h(y) \, \d y,  \]
we obtain
\[\begin{split}
\int_{\mathbb{R}_{+}^{N}} \Psi (\tilde{\Phi}_N(\phi))  \sum_{i=1}^{N} \phi_{i} \, \left( \bar{h}_{i+1} + \bar{h}_{i-1} - 2 \bar{h}_{i} \right) \, \mathbb{P}^{f}_{\varphi_{a}, N}({\rm{d}} \phi) &= \int_{\mathbb{R}_{+}^{N}} \Psi (\tilde{\Phi}_N(\phi)) N^2 \int_0^1 \tilde{\Phi}_N(\phi)_y \Delta^N_h(y) \, \d y\\
&=\int_{H} \Psi (\zeta)  \Theta^N_h(\zeta)  \, \mathbf{P}^{f}_{\varphi_{a}, N}({\rm{d}} \zeta).
\end{split}\]
The formula follows.
\end{proof}

\begin{rk}
Note that, if we choose $\varphi_a$ and $a=a_N$ as before so that $\mathbf{P}^{f}_{\varphi_{a}, N} \underset{N \to \infty}{\longrightarrow} P^1_0$, then, resaoning as in \cite[Chap. 6.6]{zambotti2017random}, we deduce that
 \[ \int_{H} \partial_{\Pi_N h} \Psi(\zeta) \, \mathbf{P}^{f}_{\varphi_{a}, N}({\rm{d}} \zeta) \underset{N \to \infty}{\longrightarrow} \int_{H} \partial_{h} \, \Psi(\zeta) \, P^1_0 ({\rm{d}} \zeta)  \]
and
\[\int_{H} \Theta^N_h(\zeta) \, \Psi(\zeta) \, \mathbf{P}^{f}_{\varphi_{a}, N}({\rm{d}} \zeta) \underset{N \to \infty}{\longrightarrow} \int_{H} \int_0^1 \zeta_r \, h''_r \d r \, \Psi(\zeta) \, P^1_0 ({\rm{d}} \zeta).  \]
Therefore, in view of equality (4.4) of Theorem 4.1 in \cite{EladAltman2019}, it would be tempting to conjecture that the first term in the right-hand side of \eqref{ibpf_pinning_path} converges as $N \to \infty$ to
\begin{equation} 
\label{second_derivative_exp}
\frac{1}{4} \int_{0}^{1} h_{r} \, \frac{{\rm d}^{2}}{{\rm d} a^{2}} \, \Sigma^1_r (\Phi(X) \, | \, a) \, \biggr\rvert_{a=0}  \d r, 
\end{equation}
where, for all $a \geq 0$,
\[\Sigma^1_r (\Phi (X) \, | \, a) := p^{1}_{r}(a) \,
 P^{1}_0 [ {\rm d} X \,| \, X_{r} = a], \]
with $p^1_r(a) := \sqrt{\frac{2}{ \pi r}} \exp \left(-\frac{a^2}{2r} \right)$ denoting the density of $X_r$ when $X$ is a reflecting Brownian motion started from $0$. Recall that, if $X$ is a reflecting Brownian, then for all $\lambda >0$, $X \overset{(d)}{=}R_{\lambda}(X)$, where $(R_{\lambda}(X))_r := \frac{1}{\sqrt{\lambda}} X_{\lambda r}$, $r \geq 0$. Hence, we have the scaling property
\[\frac{{\rm d}^{2}}{{\rm d} a^{2}} \, \Sigma^1_r (\Phi(X) \, | \, a) \, \biggr\rvert_{a=0} = \lambda^{3/2} \frac{{\rm d}^{2}}{{\rm d} a^{2}} \, \Sigma^1_{\lambda r} (\Phi \circ R_{\lambda}(X)) \, | \, a) \, \biggr\rvert_{a=0}. \]
This scaling property is compatible with the scaling factor $N^{3/2}$ appearing in the definition of $D^N_{\Psi}$.
\end{rk}

\subsection{Conjecture for the scaling limit}
\label{subsect_conjec_scaling_limit}

Above we have shown the tightness of the family of processes $(Y^N_t)_{t \in [0,T]}$, $N \geq 1$. We make the following conjecture for the corresponding, expected, limit in law. Let us denote by $(u_t)_{t \geq 0}$ the reversible Markov process associated with the Dirichlet form $\mathcal{E}$ generated by the bilinear form
\[ \mathcal{E}(f,g) := \frac{1}{2} \int \langle \nabla f , \nabla g \rangle \d \mu, \qquad f,g \in \mathcal{F} \mathcal{C}^{\infty}_{b}(K),
\]
where $\mu$ denotes the law, on $K=\{h \in H, \quad h \geq 0 \quad \text{a.e.} \}$, of a reflecting Brownian motion on $[0,1]$. Such a process can be constructed using exactly the same techniques as in Section 5 of \cite{EladAltman2019} for the case of the modulus of a Brownian bridge. We consider the process $(u_t)_{t \geq 0}$ started from equilibrium, i.e.\ $u_0 \overset{(d)}{=} \mu$. Moreover, arguing as in Theorem 5.7 of \cite{EladAltman2019}, one can show that $(u_t)_{t \geq 0}$ satisfies an equation of the form:
\[\frac{\partial u}{\partial t}=\frac 12
\frac{\partial^2 u}{\partial x^2}
 - \frac{1}{4} \, \underset{\epsilon \to 0}{\lim} \, \rho''_{\epsilon}(u) + \xi, \]
where $\xi$ is a space-time white noise on $\mathbb{R}_+ \times [0,1]$, and $(\rho_{\epsilon})_{\epsilon >0}$ is an appropriate approximation of the Dirac measure $\delta_0$.

\begin{cnj}
For all $T>0$, as $N \to \infty$, $(Y^N_t)_{t \in [0,T]}$ converges in law in $C([0,T], H^{-1}(0,1))$ to the process $(u_t)_{t \in [0,T]}$.
\end{cnj}

A natural route to prove the above conjecture would be to show that, after rescaling, the IbPF of Prop. \ref{prop_ibpf_pinning} above converges to the IbPF (4.4) of \cite{EladAltman2019} for the law of the reflecting Brownian motion, and use the same techniques as in \cite{zambotti2004fluctuations} to deduce therefrom the convergence of the associated evolutions. Unfortunately, in spite of some similarities between the two IbPF, it is not clear at all that the former converges to the latter. Another problem that arises is related with the distributional nature of the last term appearing in these IbPF. Finally, an important feature exploited in \cite{zambotti2004fluctuations} is the uniform continuity of the Markov semi-groups. In our case, this feature is non-trivial, and would in particular imply the strong Feller property for the Markov semigroup associated with $(u_t)_{t \geq 0}$, which is an open problem.

\section{A wetting model in the continuum}
\label{sec_continuum}

In this section we introduce an analog of the wetting model in the continuum, which corresponds to the law of a $3$-dimensional Bessel process tilted by a functional of its local times.

\subsection{Motivation: wetting model and local times}

Recall that, in the discrete setting described above, for the case of a wetting model with a strip, we have
\[ \mathbb{P}^{f}_{\varphi_{a}, N}({\rm{d}} \phi) = \frac{1}{Z^{+}_{\varphi_{a},N}} \exp \left( \sum_{i=1}^{N} \varphi_{a}(\phi_{i}) \right) \mathbb{P}^{+}_{N}({\rm{d}} \phi), \]
where $\mathbb{P}^{+}_{N}$ is the law of a standard Gaussian walk on $\mathbb{R}^{N}$ conditioned to remain nonnegative, and $Z^{+}_{\varphi_{a},N}$ is a normalisation constant. For all $a > 0$, $\varphi_{a}$ is a smooth function. For the sake of simplicity, consider however
\begin{equation}
\label{simple_exp}
 \varphi_{a} = \beta_{a} \mathbf{1}_{[0,a]},
\end{equation}
where the sequence $(\beta_{a})_{a > 0}$ is such that
\[ a e^{\beta_{a}} \underset{a \to 0}{\longrightarrow} e^{\beta_{c}}.  \]
Note that this condition ensures that the following convergence of measures holds in the weak sense on $\mathbb{R}_{+}$:
\[ e^{\varphi_{a}(x)} \d x \underset{a \to 0}{\longrightarrow} e^{\beta_{c}} \delta_{0}({\rm{d}} x) + \d x.\]
With the ansatz \eqref{simple_exp}, we can rewrite the wetting measure with strip as
\[ \mathbb{P}^{f}_{\varphi_{a}, N}({\rm{d}} \phi) = \frac{1}{Z^{+}_{\varphi_{a},N}} \exp \left( \beta_{a} \sum_{i=1}^{N} \mathbf{1}_{[0,a]}(\phi_{i}) \right) \mathbb{P}^{+}_{N}({\rm{d}} \phi). \]
Thus, the wetting measure corresponds to the measure $\mathbb{P}^{+}_{N}$ tilted by the local time in the strip $[0,a]$ of the random walk. We could hope that such a description be stable under taking the scaling limit: in the continuum, the law $P^{1}_{0}$ of a reflected Brownian motion started from $0$ would correspond to the law $m$ of a Brownian meander tilted by some appropriate functional of its continuous local time process. As such, this claim is false, since the probability measures $P^{1}_{0}$ is not absolutely continuous with respect to $m$. However, me may ask whether

\begin{equation}
\label{limit_representation}
P^{1}_{0}({\rm{d}} X) = \underset{\eta \to 0}{\lim} \, \exp(\Phi_{\eta} (L(X))) \, m ({\rm{d}} X),
\end{equation}
where $L(X) = \left( L^{b}_{t} (X) \right)_{b \geq 0, t \geq 0}$ is the local time process associated with $(X_{t})_{0 \leq t \leq 1}$, when $X \overset{(d)}{=} m$, and $\Phi_{\eta}, \eta >0$, are appropriate functionals on $C([0,1] \times \mathbb{R}_{+}, \mathbb{R})$. Note that we could not hope to choose $\Phi_{\eta}$ to depend solely on $L^{0}(X)$, the local time at $0$ of $X$, since this identically vanishes when $X \overset{(d)}{=} m$: we need instead to make it depend on the process $L^\eta(X)$ for $\eta$ close to $0$.

\subsection{Static approximation result}

To obtain a representation of the form \eqref{limit_representation}, we shall proceed as follows. Let $C([0,1])$ be the space of continuous function on $[0,1]$ endowed with the topology of supremum norm and the associated Borel $\sigma$-algebra. For all $a \geq 0$, we denote by $P^{3}_{a}$ (resp. $P^{1}_{a}$) the law on $C([0,1])$ of a $3$-dimensional (resp. $1$-dimensional) Bessel process started from $a$, while $m$ denotes the law of a Brownian meander. Recall that $P^{3}_{0}$ and $m$ are mutually absolutely continuous on $C([0,1])$. Moreover, under $P^{3}_{a}$, the canonical process $(X_{t})_{0 \leq t \leq 1}$ satisfies the SDE
\begin{equation}
\label{sde_3_bes}
X_{t} = a + \int_{0}^{t} \frac{\d s}{X_{s}} + B_{t}.
\end{equation}
On the other hand, under $P^{1}_{a}$, the canonical process satisfies the equation
\[X_{t} = a + \frac{1}{2} L^{0}(X)_{t} + B_{t}. \]
Hence, the idea is to approximate the latter equation by an SDE of the form
\begin{equation}
\label{sde_approx_1_bes}
X^{\eta}_{t} = a + \int_{0}^{t} f_{\eta}(X_{s}) \d s + B_{t}, 
\end{equation}
where $f_{\eta}$ is a well-chosen function such that \eqref{sde_approx_1_bes} is obtained from \eqref{sde_3_bes} through a Girsanov transform. 
Then, for any fixed $\eta >0$, one can apply Girsanov's theorem to obtain
\[ P^{1,\eta}_{a} ({\rm{d}} X) = \exp(\Phi_{\eta} (L(X))) \, P^{3}_{a} ({\rm{d}} X), \]
where $P^{1,\eta}_a$ is the law of $X^{\eta}$, and $\Phi_{\eta}$ is some functional depending on $\eta$.
This strategy is implemented in the next theorem.


\begin{thm}
\label{conv}
The convergence $P^{1,\eta}_{a} \underset{\eta \to 0}{\longrightarrow} P^{1}_{a}$ holds in the sense of weak topology for probability measures on $C([0,1])$, where, for all $\eta >0$
\[P^{1,\eta}_{a}({\rm{d}} X) = \frac{X_{1} \wedge \eta}{X_{1}} \frac{a}{a \wedge \eta} \exp \left( \frac{1}{2 \eta} L^{\eta}_{1} \right) \, P^{3}_{a} ({\rm{d}} X).\]
Here, for a $3$-Bessel process $X$, $(L^{b}_{t})_{b \geq 0, t \geq 0}$ denotes the semimartingale local time process associated with $X$, and we are using the convention $ \frac{0}{0 \wedge \eta} = 1$. In particular $P^{1,\eta}_{0} \underset{\eta \to 0}{\longrightarrow} P^{1}_{0}$, where
\[P^{1,\eta}_{0}({\rm{d}} X) = \sqrt{\frac{2}{\pi}} (X_{1} \wedge \eta) \exp \left( \frac{1}{2 \eta} L^{\eta}_{1} \right) \, m ({\rm{d}} X),\]
and where $m$ is the law of a Brownian meander on $[0,1]$.
\end{thm}

\begin{rk}
As a by-product of the proof of Theorem \ref{conv}, $P^{1,\eta}_{a}$ is actually the law, on $C([0,1])$, of the unique strong solution to the SDE
$$X_{t} = a + \int_{0}^{t} \frac{\mathbf{1}_{\{X_{s} \leq \eta\}}}{X_{s}} \d s + B_{t}, $$
which corresponds to the SDE of a $3$-Bessel process with an additional truncation in the drift. In particular, if $\eta = +\infty$, we recover the law $P^{3}_{a}$, so formally $P^{3}_{a}=P^{1,\infty}_{a}$.  Note moreover that, by  comparison (see Theorem (3.7) in \cite[Chapter IX]{revuz2013continuous}) the solution $(X_t)_{t \geq 0}$ of the above SDE is non-decreasing in $\eta$, so that the family $(P^{1,\eta}_{a})_{ \eta 
>0}$ is non-decreasing in $\eta$ for the 
stochastic ordering on $C([0,1])$ when 
the latter is endowed with the usual 
partial ordering: $u \leq v$ if $u(t) \leq v(t)$ for all $t \in [0,1]$. Finally, Theorem \ref{conv} states that the limit  $P^{1,0}_{a} := \underset{\eta \to 0}{\lim} \downarrow P^{1,\eta}_{a}$
is given by $P^{1}_{a}$. To sum up, denoting by $\preceq$ the stochastic ordering for probability measures on $C([0,1])$, we thus have 
$$P^{1}_{a} = P^{1,0}_{a} \preceq P^{1,\eta}_{a} \preceq P^{1,\eta'} \preceq P^{1,\infty}_{a} = P^{3}_{a},$$ 
for all $\eta \leq \eta'$. Thus, the family of continuous wetting measures $(P^{1,\eta}_{a})_{ \eta >0}$ provides a monotonously continuous interpolation between the laws $P^{3}_{a}$ and $P^{1}_{a}$.  
\end{rk}

\begin{rk}
Theorem \ref{conv} provides a continuous approximation of the law of a reflecting Brownian motion on the interval $[0,1]$. The proof given below actually works on any finite interval $[0,T]$, $T>0$, in which case the density in the definition of $P^{1,\eta}_{a}({\rm{d}} X)$ has to be replaced with
\begin{equation*}
\frac{X_{T} \wedge \eta}{X_{T}} \frac{a}{a \wedge \eta} \exp \left( \frac{1}{2 \eta} L^{\eta}_{T} \right),
\end{equation*}  
which (see \eqref{exp_martingale} below) is an exponential martingale in $T$. However, the result does not extend to $T=\infty$, because this martingale is not uniformly integrable, and actually converges a.s. to $0$ as $T \to \infty$. Indeed, if $X$ is a $3$-dimensional Bessel process started e.g. from $a=0$, then on the one hand $X_T \underset{T \to \infty}{\longrightarrow} + \infty$ a.s., while on the other hand, by Ex. 2.5.a. in \cite[Chapter XI.2]{revuz2013continuous}, $(L^{a}_{\infty})_{a \geq 0}$ is distributed as a $2$-dimensional squared Bessel process, so is finite. Note in particular that $\exp(\frac{1}{2 \eta} L^\eta_{\infty})$ is not integrable: remarkably, $\frac{1}{2 \eta}$ is precisely the smallest coefficient $\alpha$ such that $\exp(\alpha L^\eta_{\infty})$ is not integrable.     
\end{rk}

\begin{proof}[Proof of Theorem \ref{conv}]
The second claim follows from the first one, since, by Imhof's relation (see Exercise 4.18 in \cite[ChapterXII]{revuz2013continuous}), we have
\[P^{3}_{0}({\rm{d}} X) = \sqrt{\frac{2}{\pi}} X_{1} \, m({\rm{d}} X). \]
So it suffices to prove the first claim. Under $P^{3}_{a}$, the canonical process on $C([0,1])$ satisfies the SDE
\[ X_{t} = a + \int_{0}^{t} \frac{\d s}{X_{s}} + B_{t}.\]
We will first prove that for all $\eta >0$,  $P^{1,\eta}_{a}$ is the law, on $C([0,1])$, of the unique strong solution $X^\eta$ to the SDE
\begin{equation}
\label{approx_sde}
X_{t} = a + \int_{0}^{t} \frac{\mathbf{1}_{\{X_{s} \leq \eta\}}}{X_{s}} \d s + B_{t}.
\end{equation} 
and then we will show that $X^\eta$ converges in law to $P^{1}_{a}$ as $\eta \to 0$.
This will yield the claim.

The first point is proven using Girsanov's Theorem. Indeed, consider the local martingale
\[ M_{t} = - \int_{0}^{t} \frac{\mathbf{1}_{\{X_{s} > \eta\}}}{X_{s}} d B_{s}, \quad t \geq 0. \]
The corresponding exponential local martingale is given by
\[ \mathscr{E}(M)_{t} = \exp \left( M_{t} - \frac{1}{2} \langle M, M \rangle_{t} \right), \quad t \geq 0.\]
Since 
$$\langle M, M \rangle_{t} = \int_0^t \frac{\mathbf{1}_{\{X_{s} > \eta\}}}{X_{s}^2} \, \d s \leq t \, \eta^{-2},$$ 
it follows by Corollary (1.16) in \cite[Chapter VIII]{revuz2013continuous} that $\left(\mathscr{E}(M)_{t}\right)_{t \geq 0}$ is a martingale. Now, we have
\[ M_{t} - \frac{1}{2} \langle M, M \rangle_{t} = - \int_{0}^{t} \frac{\mathbf{1}_{\{X_{s} > \eta\}}}{X_{s}} \d B_{s} - \frac{1}{2} \int_{0}^{t} \frac{\mathbf{1}_{\{X_{s} > \eta\}}}{X_{s}^{2}} \d s.\]
We intend to re-express this quantity without stochastic integral. To do so,  consider the function $F: \mathbb{R}_{+}^{*} \to \mathbb{R}_{+}$ defined by
\[ F(x) := \log \left( \frac{x \wedge \eta}{x} \right), \quad x > 0.\]
$F$ is the difference of two convex functions on $\mathbb{R}_{+}^{*}$. Therefore, by the It\^o-Tanaka formula (cf. Theorem (1.5) in \cite[Chap. VI]{revuz2013continuous}), we have
\[ F(X_{t}) = F(a) + \int_{0}^{t} F'(X_{s}) \d X_{s} + \frac{1}{2} \int_{\mathbb{R}} F''({\rm{d}} x) L^{x}_{t}. \]
Since
\[ F'(x) = - \frac{\mathbf{1}_{\{x > \eta\}}}{x},  \quad x > 0, \]
and
\[ F''({\rm{d}} x) = \frac{\mathbf{1}_{\{x > \eta\}}}{x^{2}} \d x - \frac{1}{\eta} \delta_{\eta}({\rm{d}} x), \]
under the law $P^{3}_{a}$, the canonical process thus satisfies
\[ \log \left( \frac{X_{t} \wedge \eta}{X_{t}} \right) = \log \left( \frac{a \wedge \eta}{a} \right) - \int_{0}^{t} \frac{\mathbf{1}_{\{X_{s} > \eta\}}}{X_{s}} \left( \frac{1}{X_{s}} \d s + \d B_{s} \right) + \frac{1}{2} \int_{0}^{t} \frac{\mathbf{1}_{\{X_{s} > \eta\}}}{X_{s}^{2}} \d s - \frac{1}{2 \eta} L^{\eta}_{t},\]
whence we obtain
\[ - \int_{0}^{t} \frac{\mathbf{1}_{\{X_{s} > \eta\}}}{X_{s}} \d B_{s} - \frac{1}{2} \int_{0}^{t} \frac{\mathbf{1}_{\{X_{s} > \eta\}}}{X_{s}^{2}} \d s = \log \left( \frac{X_{t} \wedge \eta}{X_{t}} \frac{a}{a \wedge \eta} \right) +  \frac{1}{2 \eta} L^{\eta}_{t}. \]
Therefore
\begin{equation}
\label{exp_martingale} 
\mathscr{E}(M)_{t} = \frac{X_{t} \wedge \eta}{X_{t}} \frac{a}{a \wedge \eta} \exp \left( \frac{1}{2 \eta} L^{\eta}_{t} \right).
\end{equation}
By Girsanov's theorem, under the probability law $\mathscr{E}(M)_{1} \, P^{3}_{a}=P^{1,\eta}_{a}$ on $C([0,1])$, the canonical process satisfies \eqref{approx_sde}. Moreover, this SDE has pathwise uniqueness since the function $x \mapsto \frac{\mathbf{1}_{\{x \leq \eta\}}}{x}$ is non-increasing on $\mathbb{R}_+$. Therefore, by the Yamada-Watanabe Theorem, the SDE \eqref{approx_sde} admits a unique strong solution $(X^{\eta})_{t \geq 0}$, the restriction of which to $[0,1]$ has law $P^{1,\eta}_{a}$.
There remains to establish the convergence in law of $(X^{\eta})_{t \geq 0}$ to $P^{1}_{a}$. To do so, 
note that by the comparison theorem (3.7) in \cite[Chapter IX]{revuz2013continuous},
for all $\eta, \bar{\eta} >0$ such that $\eta \leq \bar{\eta}$, we have $X^{\eta} \leq X^{\bar{\eta}}$ a.s. 
Since moreover $X^{\eta} \geq 0$ a.s. for all $\eta$, we deduce the existence of a process $X_{t} \geq 0$ such that, for any decreasing sequence of positive real numbers $(\eta_n)_{n \geq 1}$ converging to $0$, the sequence $(X^{\eta_n})_{n \geq 1}$ converges a.s. pointwise on $\mathbb{R}_+$ to $X$ from above. Therefore, the only possible subsequential weak limit of $P^{1,\eta}_{a}$ as $\eta \to 0$ is given by the law of $(X_t)_{0 \leq t \leq 1}$ on $C([0,1])$.
Setting
\[ Z_{t} := X_{t}^{2}, \quad t \geq 0 \]
and, for all $\eta >0$
\[ Z^{\eta}_{t} := (X^{\eta}_{t})^{2}, \quad t \geq 0, \]
then, by It\^{o}'s lemma
\begin{equation}
\label{sde_z_eta}
 Z^{\eta}_{t} = a^{2} + 2 \int_{0}^{t} \sqrt{Z^{\eta}_{s}} \d B_{s} + 2 \int_{0}^{t} \mathbf{1}_{\{Z^{\eta}_{s} \leq \eta^{2}\}} + t.
\end{equation}
From this equation, we deduce that the sequence of probability measures $(P^{1,\eta}_{a})_{\eta > 0}$ is tight on $C([0,1])$. Indeed, by \eqref{sde_z_eta}, for all $0 \leq s < t \leq 1$
\[\mathbb{E} \left[ |Z^{\eta}_{t}-Z^{\eta}_{s}|^{4} \right] \leq C \left( \left(\int_{s}^{t}  \mathbb{E}(Z^{\eta}_{u}) \d u \right)^{2} + (t-s)^{2} \right),\]
where $C>0$ is some universal constant. Since, by comparison, $Z^{\eta} \leq Z^{\infty}$, where $Z^{\infty}$ is a $3$-dimensional squared Bessel process, we deduce that
\[ \mathbb{E} \left[ |Z^{\eta}_{t}-Z^{\eta}_{s}|^{4} \right] \leq C'(t-s)^{2}, \]
for some (other) universal constant $C'>0$, whence
\[ \mathbb{E} \left[ |X^{\eta}_{t}-X^{\eta}_{s}|^{8} \right] \leq C'(t-s)^{2},\]
and the claimed tightness follows by Kolmogorov's tightness criterion, see Theorem (1.8) in \cite[Chapter XIII]{revuz2013continuous}. Hence $P^{1,\eta}_{a}$ converges weakly as $\eta \to 0$ to the law of $(X_t)_{0 \leq t \leq 1}$ on $C([0,1])$. We finally show that the latter is given by $P^{1}_{a}$. By the comparison theorem (3.7) in \cite[Chapter IX]{revuz2013continuous}, we deduce from \eqref{sde_z_eta} that, a.s., for all $n \geq 1$, $Z^{1/n} \geq Z^{0}$, where $Z^{0}$ is the unique strong solution of
\[ Z^{0}_{t} = a^{2} + 2 \int_{0}^{t} \sqrt{Z^{0}_{s}} \d B_{s} + t. \]
Sending $n \to \infty$, we deduce that, a.s., $Z \geq Z^{0}$.
Note that $Z^{0}$ is a one-dimensional squared Bessel process. Therefore, almost-surely, $Z_{t} \geq Z^{0}_{t} > 0$ for a.e. $t \geq 0$. Hence, a.s., for a.e. $t \geq 0$
\[ \mathbf{1}_{\{Z^{\eta}_{t} \leq \eta^{2}\}} \underset{\eta \to 0}{\longrightarrow} 0. \]
Hence, by dominated convergence, letting $\eta \to 0$ in \eqref{sde_z_eta}, we deduce that $Z$ satisfies the SDE
\[Z_{t} = a^{2} + 2 \int_{0}^{t} \sqrt{Z_{s}} \d B_{s} + t. \]
By strong uniqueness of this SDE, we deduce that $Z=Z^{0}$. Hence, in particular, the law of $(X_t)_{0 \leq t \leq 1}$ on $C([0,1])$ is given by $P^{1}_{a}$, and 
the proof is complete.
\end{proof}

\subsection{The corresponding dynamics}
In order to study the gradient dynamics of the measures $P^{1,\eta}_0$, we introduce a further mollified version of these. Namely, for all $\eta>0$ and $\epsilon \in (0,\eta)$, let
\begin{equation}
\label{mollified_cont_wetting}
P^{1,\eta,\epsilon}_0 = \frac{1}{Z_{\eta,\epsilon}} \frac{X_{1} \wedge \eta}{X_{1}} \frac{a}{a \wedge \eta} \exp \left( \frac{1}{2 \eta} \int_0^1 \rho_{\epsilon}(X_s -\eta) \, \d s \right) \, P^{3}_{0} ({\rm{d}} X),
\end{equation}
where $Z_{\eta, \epsilon}$ is a normalization constant, and $\rho_\epsilon = \frac{1}{\epsilon} \rho(\frac{x}{\epsilon})$, with $\rho$ a smooth, even function supported in $[-1,1]$ such that
$$\rho \geq 0, \qquad \int_{\mathbb{R}} \rho =1, \qquad \rho' \leq 0 \quad \text{on} \, \, \mathbb{R}_+. $$

\begin{prop}
For all $\eta>0$, $P^{1,\eta,\epsilon}_0$ converges weakly to $P^{1,\eta}_0$ as $\epsilon \to 0$.
\end{prop}

\begin{proof}
It suffices to show that, for all bounded, continuous function $F$ on $C([0,1])$, we have
\begin{equation}
\label{weak_conv}
E^{3}_{0} \left[\frac{X_{1} \wedge \eta}{X_{1}} \exp \left( \frac{1}{2 \eta} \int_0^1 \rho_{\epsilon}(X_s -\eta) \, \d s \right)  F(X) \right] \underset{\epsilon \to 0}{\longrightarrow} E^{3}_{0} \left[ \frac{X_{1} \wedge \eta}{X_{1}} \exp \left( \frac{1}{2 \eta} L^{\eta}_1 \right) F(X)  \right],
\end{equation}
where we have denoted by $E^3_0$ the expectation operator associated with $P^3_0$. By the occupation times formula,  
$$\int_0^1 \rho_{\epsilon}(X_s -\eta) \, \d s = \int_0^\infty \rho_{\epsilon}(a -\eta) L^a_1 \, d a, $$
and, by continuity of the process $(L^a_1)_{a \geq 0}$, this converges a.s. to $L^\eta_1$ as $\epsilon \to 0$. 
Moreover, defining for all $\epsilon>0$
$$R_\epsilon(x)= \int_0^x \int_{-\infty}^y \rho_{\epsilon}(z-\eta) \, \d z \, \d y, $$
since $R_{\epsilon}''(x) =\rho_{\epsilon}(x -\eta)$ and $R_\epsilon(0)=0$, we have by It\^{o}'s lemma 
$$\frac{1}{2} \int_0^1 \rho_{\epsilon}(X_s -\eta) \, \d s = R_{\epsilon}(X_1) - \int_{0}^1 R_{\epsilon}'(X_s) (\d B_s + \frac{1}{X_s} \, d s) \leq X_1 - \int_{0}^1 R_{\epsilon}'(X_s) \d B_s,  $$
where we used the fact that $R_{\epsilon}(x) \leq x$ and $R_{\epsilon}'(x) \geq 0$ for all $x$ in the last inequality. We thus obtain
$$\frac{X_{1} \wedge \eta}{X_{1}} \exp \left( \frac{1}{2 \eta} \int_0^1 \rho_{\epsilon}(X_s -\eta) \, \d s \right)  F(X) \leq \|F\|_{\infty} \exp \left(\frac{X_1}{\eta} - M^\epsilon_{1} \right),$$
where $M^\epsilon_{t} = \frac{1}{\eta}\int_{0}^t R_{\epsilon}'(X_s) \d B_s$ is a local martingale. To obtain the requested convergence, it therefore suffices to show that the random variables 
$$ \exp \left(\frac{X_1}{\eta} - M^\epsilon_{1} \right)$$
are uniformly integrable in $\epsilon$.
By the Cauchy-Schwarz inequality, this will follow upon showing that a) $\exp (2 X_1/ \eta) $ is integrable and b) $\exp(-2M^{\epsilon}_1)$ is uniformly integrable. The claim a) follows by noting (see \cite[Chapter XI]{revuz2013continuous}) that under $P^3_0$, $X \overset{(d)}{=} \sqrt{B_1^2 + B_2^2 + B_3^2}$, where $(B_1,B_2,B_3)$ is a standard Brownian motion in $\mathbb{R}^3$, so that
\begin{align*}
E^3_0 \exp (2X_1/ \eta)  = \mathbb{E} \exp \left( 2 \eta^{-1} \sqrt{B_1^2 + B_2^2 + B_3^2}\right) &\leq \mathbb{E} \exp (2 \eta^{-1} (|B_1| + |B_2| + |B_3|)) \\
&= \left(\mathbb{E} \exp (2 \eta^{-1} |B_1|) \right)^3 < \infty. 
\end{align*}
On the other hand, the claim b) follows by Novikov's criterion upon noting that
$$ \langle  2 M^\epsilon \rangle_1 = \frac{4}{\eta^2} \int_{0}^1 R_{\epsilon}'(X_s)^2 \, \d s \leq \frac{4}{ \eta^2},$$
where the last quantity depends on $\eta$ but \textit{not} on $\epsilon$. 
The requested convergence follows.
\end{proof}
Henceforth, we fix $\eta, \epsilon >0$. We recall the definition $H:= L^{2}([0,1])$, and
\[K= \{ h \in L^{2}([0,1]), \, h \geq 0 \}.\]
Let $\mathcal{F} \mathcal{C}^{\infty}_{b}(H)$ denote the space of functionals $F:H \to \mathbb{R}$ of the form
\begin{equation}\label{Fexp}
F (z) = \psi(\langle l_{1}, z \rangle, \ldots, \langle l_{m}, z \rangle ), \quad z \in H,
\end{equation}
with $m \in \mathbb{N}$, $\psi \in C^{\infty}_{b}(\mathbb{R}^{m})$, and $l_{1}, \ldots, l_{m} \in \text{Span} \{ e_{k}, k \geq 1 \}$, where, for all $n \geq 1$, $e_n$ is defined by \eqref{def_e_n}. We also define:
\[\mathcal{F} \mathcal{C}^{\infty}_{b}(K) := \left\{ F \big \rvert_{K} \, , \ F \in \mathcal{F} \mathcal{C}^{\infty}_{b}(H) \right\}.
\]
Moreover, for $f \in \mathcal{F} \mathcal{C}^{\infty}_{b}(K)$ of the form $f=F \big \rvert_{K}$, with $F \in \mathcal{F} \mathcal{C}^{\infty}_{b}(H)$, we define $\nabla f : K \to H$ by
\[ \nabla f (z) = \nabla F(z), \quad z \in K, \]
where this definition does not depend on the choice of $F \in \mathcal{F} \mathcal{C}^{\infty}_{b}(H)$ such that $f=F \big \rvert_{K}$.
%
%
With these definitions at hand, we consider the form $\mathcal{E}^{1,\eta, \epsilon}$ defined on $\mathcal{F} \mathcal{C}^{\infty}_{b}(K)$ as follows:
\[ \mathcal{E}^{1,\eta,\epsilon} (u,v) = \frac{1}{2} \int_{K} \nabla u (x) . \nabla v (x) \d P^{1,\eta,\epsilon}_0(x), \quad u,v \in \mathcal{F} \mathcal{C}^{\infty}_{b}(K). \]
By Theorems 5 and 6 in \cite{zambotti2002integration}, the form $(\mathcal{F} \mathcal{C}^{\infty}_{b}(K), \mathcal{E}^{1,\eta, \epsilon})$ is closable, and it holds that its closure $(\text{Dom}(\mathcal{E}^{1,\eta,\epsilon}), \mathcal{E}^{1,\eta, \epsilon})$ is a quasi-regular Dirichlet form. Moreover, in view of the relation \eqref{mollified_cont_wetting}, the $K$-valued Markov process $(u^{\eta,\epsilon}_{t})_{t \geq 0}$ associated with $\mathcal{E}^{1,\eta,\epsilon}$ satisfies the following Nualart-Pardoux type equation:
\begin{equation}
\label{spde_reflect_sticky_eta_epsilon}
\left\{ \begin{array}{ll}
{\displaystyle
\frac{\partial u^{\eta,\epsilon}}{\partial t} = \frac{1}{2} \frac{\partial^{2} u^{\eta,\epsilon} }{\partial x^{2}}+ \xi + \zeta + \frac{1}{4 \eta} \rho_{\epsilon}'(u^{\eta,\epsilon}-\eta)},
\\ \\
u^{\eta,\epsilon} \geq 0, \quad \d \zeta \geq 0, \quad \int_{\mathbb{R}_{+} \times [0,1]} u^{\eta,\epsilon} \d \zeta = 0,
\end{array} \right.
\end{equation}
where $\rho_\epsilon$ is as in \eqref{mollified_cont_wetting} above. Note that, for all $\epsilon \in (0,\eta)$, the term $\rho_{\epsilon}'(u^{\eta,\epsilon}-\eta)$ is zero, except when $|u^{\eta,\epsilon}-\eta | \leq \epsilon$. Moreover, it is positive when $\eta - \epsilon \leq u^{\eta,\epsilon} \leq \eta$, and negative when $\eta \leq u^{\eta,\epsilon} \leq \eta + \epsilon$. Thus, equation \eqref{spde_reflect_sticky_eta_epsilon} can be interpreted as an SPDE with reflection at $0$ and an attractive mechanism at $\eta$ which encourages the solution to remain in the interval $(\eta-\epsilon, \eta + \epsilon)$. We now formulate a conjecture for the limit in law as $\epsilon, \eta \to 0$ of the above processes.

%
%


For all $\eta, \epsilon >0$, we consider the $K$-valued stationary Markov process $(u^{\eta,\epsilon}_{t})_{t \geq 0}$ associated with the Dirichlet form $\mathcal{E}^{1,\eta,\epsilon}$, and started from equilibrium: $u^{\eta,\epsilon}_0 \overset{(d)}{=} P^{1,\eta,\epsilon}_0$. Recall that $H^{-1}(0,1)$ denotes the completion of $H=L^{2}(0,1)$ w.r.t.\ the norm
\[ \|f \|^{2}_{-1} := \sum_{n=1}^{\infty} n^{-2} \, | \langle f, e_{n} \rangle |^{2}, \]
where $e_{n}(\theta) :=\sqrt{2} \sin(n \pi \theta), \, \theta \in [0,1]$.

\begin{cnj}
\label{cnj_cv_continuous}
For all $T>0$, $(u^{\eta,\epsilon}_{t})_{t \in [0,T]}$ converges weakly to $(u_t)_{t \in [0,T]}$ in $C([0,T], H^{-1}(0,1))$ as we send $\epsilon$, then $\eta$ to $0$, where $u$ is the Markov process considered in Section \ref{subsect_conjec_scaling_limit} above.
\end{cnj}

Note that, since $P^{1,\eta,\epsilon}_0$ converges to $P^1_0$ as we send $\epsilon$, then $\eta$ to $0$, one can obtain the tightness of the whole process using the Lyons-Zheng argument, as in the proof of Theorem \ref{thm_tightness_discrete_wetting} above. To obtain the conjecture \ref{cnj_cv_continuous}, the main problem is therefore to identify the limit in probability of $u^{\eta,\epsilon}$ as $\epsilon, \eta \to 0$. For the same reasons as mentioned in Section \ref{subsect_conjec_scaling_limit} above this question is, for now, open.

\section{Acknowledgements}\label{sec:acknowledgements}
The research of J.D.D.\ and T.O.\ is supported by the German Research Foundation through the research unit
FOR 2402 -- Rough paths, stochastic partial differential equations and related topics. The authors are very grateful
to the latter for also supporting the stay of H.E.A.\ in Berlin, during which the present work was initiated. They would also like to thank Nicolas Perkowski and Lorenzo Zambotti for numerous precious discussions on this subject.

\bibliographystyle{alpha}
\bibliography{Dynamics_wetting_revised}

\end{document}